\documentclass[a4paper,12pt]{article}
\usepackage{times, url}
\usepackage{amsmath}
\usepackage{amsthm}
\usepackage{graphicx}
\usepackage{amssymb}
\usepackage{amsthm}
\usepackage{mathtools}
\usepackage{ragged2e}
\makeatletter
\usepackage{mathtools}

\DeclarePairedDelimiter\floor{\lfloor}{\rfloor}
\def\Ddots{\mathinner{\mkern1mu\raise\p@
\vbox{\kern7\p@\hbox{.}}\mkern2mu
\raise4\p@\hbox{.}\mkern2mu\raise7\p@\hbox{.}\mkern1mu}}
\makeatother

\DeclarePairedDelimiterX{\norm}[1]{\lVert}{\rVert}{#1}

\usepackage{amsmath,amssymb,amsthm,amsfonts}

\newtheorem{theorem}{Theorem}[section]

\begin{document}
\setcounter{page}{1}

\begin{center}
{\LARGE \bf Convex H\"older bound and its applications}
\vspace{2.5mm}
\\ Hariprasad M\\
hariprasadm@iisc.ac.in \\
IISc Bengaluru, India -560012

\end{center}
\vspace{2.5mm}
\begin{abstract}
	Given $l,m,s \in [1, \infty]$ such that $l < s < m$ an upper bound on the $s$ norm is given using $l$ norm and $m$ norm. The result is applied in bounding odd values of zeta function, binomial sums and gamma and beta functions. 
\end{abstract}
\noindent
\section{Introduction}
H\"older's inequality is widely applied in differential equations, analysis, statistics etc. It has been extended from the stability perspective \cite{7}, refined and sharpened \cite{3} \cite{4}, in continuous form \cite{5} and in several dimensions \cite{6}.  Here we give the convex H\"older bound, with $l,m,s \in [1, \infty]$ such that $l < s < m$ give an upper bound on the $s$ norm in terms of $l$ norm and $m$ norm. 

Prof. Dinu Teodorescu posed the following question on the research gate :  Given $a,b,c >0$ with $a+b+c = 1$, to prove $(ab)^{5/4} + (bc)^{5/4}+ (ca)^{5/4} \leq \frac{1}{4}$.  \\
There are a few ways to prove this inequality, like using elementary AM-GM inequalities, using the method of Lagrange multipliers etc. But proving this inequality by repeated applications of H\"older's inequality (\cite{1} chapter 7 p.139) motivated us to give bound on $s$ norm given $l$ and $m$ norms when $l< s < m$ ($l,s,m \in [1,\infty]$).  \\

We discuss some applications of these results on bounding odd integer values of zeta function, evaluating binomial sums, and bounding beta and gamma functions.   

\section{Results}
For $a+b+c = 1$,
we have 
\begin{align}
(a+b+c)^2 &= a^2 + b^2 +c ^2 + 2 (ab+ bc + ca), \\
ab+bc+ca & \leq 1 - \min \{a^2 + b^2 + c^2 \}, \\
ab + bc + ca & \leq \frac{1}{2}(1 - \frac{1}{3}) = \frac{1}{3}.
\end{align}

\textbf{H\"older's inequality} (\cite{1} chapter 7 p.139) Let $f,g$ be the functions on a measure space, then $\norm{fg}_1 \leq \norm{f}_p \norm{g}_q$ with $\frac{1}{p} + \frac{1}{q} = 1$ and  $p \in [1, \infty]$.
Here we use the counting measure, i.e. the discrete case of H\"olders inequality. 
\subsection{Dinu's inequality}

We repeatedly apply H\"olders inequality on  $(ab)^{5/4} + (bc)^{5/4}+ (ca)^{5/4}$ with different choices of $f_i,g_i$ and $p$. \\
Step1 : $f_1 = \begin{bmatrix}ab &bc& ca \end{bmatrix}$ and 
$g_1 = \begin{bmatrix} ab^{1/4}& bc^{1/4} & ca^{1/4} \end{bmatrix}$, we choose $q = 4$, which leads $p = \frac{4}{3}$.
This gives inequality
\begin{align}
\norm{f_1g_1} & \leq \norm{f_1}_{4/3} \norm{g_1}_4, \\
& \leq \norm{f_1}_{4/3} \left(\frac{1}{3} \right)^{1/4}. 
\end{align}  
Step 2 : Now we chose $f_2 = \begin{bmatrix}ab &bc& ca \end{bmatrix}$ and $g_2 = \begin{bmatrix} ab^{1/3}& bc^{1/3} & ca^{1/3} \end{bmatrix}$ such that $(\norm{f_2g_2}_1)^{3/4} = \norm{f_1}_{4/3}$.
Further we have $q = 3$ which yields $p = 3/2$.
\begin{align}
\norm{f_2g_2}_1 \leq \norm{f_2}_{3/2} \left( \frac{1}{3} \right)^{1/3}.   
\end{align}

Step 3 : Now we chose  $f_3 = \begin{bmatrix}ab &bc& ca \end{bmatrix}$ and $g_3 = \begin{bmatrix} ab^{1/2}& bc^{1/2} & ca^{1/2} \end{bmatrix}$ such that $(\norm{f_3g_3}_1)^{2/3} = \norm{f_2}_{3/2}$.
We chose $q = 2$ in this case which leads us $p = 2$.
\begin{align}
\norm{f_3g_3}_1 \leq \norm{f_3}_{2} \left( \frac{1}{3} \right)^{1/2}.   
\end{align}

Step 4 : Now we chose  $f_4 = \begin{bmatrix}ab &bc& ca \end{bmatrix}$ and $g_4 = \begin{bmatrix} ab& bc& ca \end{bmatrix}$ and we chose $q = 1$, this leads $p = \infty$.
This gives,
\begin{align}
\norm{f_4g_4}_1 \leq \frac{1}{3} \frac{1}{4}.
\end{align}

Now doing back substitutions,
\begin{align}
\norm{f_3g_3}_1 &\leq \left( \frac{1}{3} \frac{1}{4} \right)^{1/2} \left( \frac{1}{3} \right)^{1/2} \\
 &= \frac{1}{3}  \left(\frac{1}{4} \right)^{1/2}
\end{align}
This gives, 
\begin{align}
\norm{f_2g_2}_1 &\leq  \left(\frac{1}{3}  \left(\frac{1}{4} \right)^{1/2} \right)^{2/3} \left( \frac{1}{3} \right)^{1/3}. \\
& = \frac{1}{3} \left(\frac{1}{4} \right)^{1/3}. 
\end{align}
Further back substitution gives,

\begin{align}
\norm{f_1g_1}_1 &\leq \left(\frac{1}{3} \left(\frac{1}{4} \right)^{1/3} \right)^{3/4}\left(\frac{1}{3} \right)^{1/4} \\
&= \frac{1}{3} \left(\frac{1}{4} \right)^{1/4} = 0.23570 < 1/4.
\end{align}

In general if 
$\sum_{k=1}^na_k = 1$ and $a_k > 0$ Then 
\begin{align*}
{\displaystyle \prod\limits_{i=1}^{n}a_i^{1+\frac{1}{m}}\sum_{k = 1}^n \frac{1}{a_k}^{1+\frac{1}{m}} \leq \frac{n-1}{2n} \left(\frac{1}{4}^{\frac{1}{m}}\right)}.
\end{align*} 

This recursive splitting and back substitution can be looked at as the following theorem.
\subsection{Convex H\"older bound}
\begin{theorem}
Given $l < s < m$ and given $l$ norm and $m$ norm of a function $f$ in a measure space, we have 
$(\norm{f}_s)^s  \leq (\norm{f}_l)^{l\frac{(m-s)}{m-l}} (\norm{f}_m)^{m\frac{s-l}{m-l}}$.
\end{theorem}
\begin{proof}
we split $|f|^s = |f|^{\alpha + \beta}$ and apply the H\"olders inequality with a suitable choice of $p$. In order to use $l$ and $m$ norms, we have the following conditions,
\begin{align}
\begin{bmatrix}
1 & 1 \\
p & 0  \\
0 & \frac{p}{p-1} 
\end{bmatrix}
\begin{bmatrix}
\alpha \\
\beta 
\end{bmatrix} 
= \begin{bmatrix}
s \\
l \\
m
\end{bmatrix}.
\end{align}
From $\alpha + \beta = s$, we get the following convex combination, $l\frac{1}{p} + m \frac{p-1}{p} = s$.
Thus solving for $p$, 
\begin{align}
p &= \frac{m-l}{m-s}, \\
q &= \frac{p}{p-1},  \\
&= \frac{m-l}{s-l}.
\end{align}

Now by applying the H\"olders inequality we get,
\begin{align}
(\norm{f}_s)^s  \leq (\norm{f}_l)^{l\frac{(m-s)}{m-l}} (\norm{f}_m)^{m\frac{s-l}{m-l}}. \label{lcb}
\end{align} 

When $m = \infty$, we can chose $p = 1$. This gives $\alpha = l$ and $\beta = s-l$. This gives,

\begin{align}
(\norm{f}_s)^s  \leq (\norm{f}_l)^{l} (\norm{f}_\infty)^{s-l}. 
\end{align}
\end{proof}

From the above theorem, the bound obtained in the previous section is given when $l = 1, m= \infty$ and $s = \frac{5}{4}$. 

\section{Applications}
\subsection{Bounding $\zeta(2k+1)$}
Riemann zeta function $\zeta(s) = \sum_{n=1}^{\infty} \frac{1}{n^s}$ has nice closed form expressions at even integers $s$ ( from \cite{2} Theorem 12.17), The bound in equation \eqref{lcb} can be used to bound the value of $\zeta(s)$ for odd $s$. 
Consider $f = \{ \frac{1}{ n^2} \}$, then $\norm{f}_1 = \frac{\pi^2}{6}$. 
Let us bound $\sum_{n=1}^{\infty} \frac{1}{n^3} = (\norm{f}_{3/2})^{3/2}$.
From equation \eqref{lcb}, let $l =1, m =2$, 
we get 
 \begin{align}
 \sum_{n=1}^{\infty} \frac{1}{n^3} \leq \left(\frac{\pi^2}{6} \right)^{\frac{1}{2}}\left(\frac{\pi^4}{90} \right)^{\frac{1}{2}} = \frac{\pi^3}{6 \sqrt{15}}.
 \end{align}
 In general $\zeta(2k+1) = (\norm{f}_{\frac{2k+1}{2}})^{\frac{2k+1}{2}}$, by using $k$ norm and $k+1$ norm we get,
 
 \begin{align}
 \zeta(2k+1) \leq \sqrt{\zeta(2k)\zeta(2k+2)}. 
 \end{align}
 The Table \ref{zetat} shows the approximate value of $\zeta(2k+1)$ and its bound. 
 
 \begin{table}
 	\begin{tabular}{|l|l|l|l|}
 		\hline 
 		$k$ &  $\zeta(k)$ & bound  $\sqrt{ \zeta(\frac{k-1}{2}) \zeta(\frac{k+1}{2})}$ & closed form of bound \\ \hline
 		3& 1.20205690315959  &  1.33429770234112 & $\frac{\pi^3}{6 \sqrt{15}}$ \\ \hline
 		5& 1.03692775514336	&	1.04933027814916 & $\frac{\pi^5}{45 \sqrt{42}}$ \\ \hline
 		7  & 1.00834927738192 &1.01068844458798  & $\frac{\pi^7}{945 \sqrt{10}}$\\ \hline
 		9  & 1.00200839282608  &1.00253478072475 &$\frac{\pi^9}{2835 \sqrt{110}} $ \\ \hline
 		11  & 1.00049418860411 &1.00062026085458 & $\frac{\pi^{11} \sqrt{691}}{5 \sqrt{273}}$ \\ \hline
 		\end{tabular}
 	\caption{Value of $\zeta(k)$, its upper bounds numerical value and the closed form expression.} \label{zetat}
 \end{table}
\subsection{Bounding binomial sums}
Binomial sums of the form $\sum_{k=0}^N \binom{N}{k} k^s$ is possible to evaluate recursively when $s$ is a non negative integer. However when $s > 1 $ is not an integer, the expression does not have the closed form. as an example, we have $\sum_{k=0}^N \binom{N}{k} k = N 2^{N-1}$ and
 $\sum_{k=0}^N \binom{N}{k}k^2 = 2^{N-2} (N+N^2)$.
 Thus for any $1 \leq s \leq 2$ we have
 \begin{align}
 \sum_{k=0}^N \binom{N}{k} k^s &\leq (N^{2-s} (N+N^2)^{s-1})  2^{(N-1)(2-s)+ (N-2)(s-1)}, \\
  \sum_{k=0}^N \binom{N}{k} k^s &\leq (N^{2-s} (N+N^2)^{s-1})  2^{N-s}. 
 \end{align} 
 
\subsection{Bounding integrals and special function}

Consider the integral $\int_{0}^{\pi/2} \sin(x)^{\frac{3}{2}}dx$, which has the numerical value $\beta(\frac{1}{2},\frac{5}{4}) = 0.87401918476404$, by using the \eqref{lcb} with $l=1$ and $m=2$, we get 
\begin{align}
\int_{0}^{\pi/2} \sin(x)^{\frac{3}{2}}dx &\leq \left(\int_{0}^{\pi/2} \sin(x)dx \right)^{\frac{1}{2}}  \left(\int_{0}^{\pi/2} \sin(x)^{2}dx \right)^{\frac{1}{2}}, \\
& =\sqrt{\frac{\pi}{2}} = 0.88622692545276.
\end{align}  

The Gamma function has the integral form 
$\Gamma(y+1) = \int_0^\infty x^{y} e^{-x} dx$. 
Given $ y >1$, and $l$ and $m$ are the nearest integers such that $l<y<m$ and $m=l+1$.
Further we chose the measure space such that $\norm{x}_y = \Gamma(y+1)$.
Using \eqref{lcb} we have,
\begin{align*}
\Gamma(y+1) \leq \Gamma(l+1)^{l+1-y} \Gamma(l+2)^{y-l}.
\end{align*} 
Further we have $\Gamma(l+1) =  l!$. This gives
\begin{align}
\Gamma(y+1) \leq (l!) (l+1)^{y-l}.
\end{align}

We have the expression for $\beta$ function as
\begin{align}
\beta(x,y) = \frac{\Gamma(x) \Gamma(y)}{\Gamma(x+y)}.
\end{align}
With $\floor{x} = l \geq 1$ and $\floor{y} = m \geq 1$, and $\floor{x+y} = k$  we get, 
\begin{align}
\beta(x+1,y+1) \leq \frac{(l! m!) (l+1)^{x-l} (m+1)^{y-m}}{(l+m)!}.  
\end{align}

\textbf{Acknowledgement :} Author thanks Dr. Dinu Teodorescu for useful discussions.

\end{document}